\documentclass[11pt,reqno]{amsart} 
\usepackage{amsfonts}
\usepackage{mathrsfs}
\usepackage{graphicx}
\usepackage{amssymb,latexsym}
\usepackage{cite} 
\usepackage[height=190mm,width=130mm]{geometry} 
\usepackage{stackengine}
\setstackgap{S}{1pt}

\theoremstyle{plain}
\newtheorem{theorem}{Theorem}
\newtheorem{lemma}{Lemma}
\newtheorem{corollary}{Corollary}

\theoremstyle{definition}
\newtheorem{definition}{Definition}

\theoremstyle{remark}
\newtheorem{remark}{Remark}
\newtheorem{example}{Example}

\numberwithin{equation}{section} 
\def\mbf#1{\mbox{\boldmath$#1$}}
\newcommand{\R}{\mathbb{R}}

\begin{document}
\title[
\hspace{-2mm}  Application of Divergent Series to particular solutions of ODE]
{Note on the Application of Divergent Series \\for Finding a Particular Solution\\
to a Nonhomogeneous Linear Ordinary Differential Equation  with Constant Coefficients}
\author{Jozef Fecenko}
\address{University of Economics\\ Faculty of Economic Informatics\\ Department of Mathematics and Actuarial Science\\  Dolnozemsk\'a St.\\ 832 04 Bratislava\\ Slovakia}
\email{jozef.fecenko@euba.sk}

\thanks{This work was supported by the Slovak Grant Agency VEGA No. 1/0647/19}

\begin{abstract}
There are many methods for finding a particular solution to a nonhomogeneous linear ordinary differential equation (ODE) with constant coefficients. The method of undetermined coefficients, Laplace transform method and differential operator method are generally known. The latter mentioned method sometimes uses the Maclaurin expansion of an inverse differential operator but only in the case when the obtained series is convergent. The present work deals also with how to find a particular solution if  the corresponding infinite series is divergent using only the terms of that series and the method of summation of divergent series.
\end{abstract}


\maketitle

\section{Introduction}

The history of divergent series goes back to L. Euler, who had an idea that any divergent series should have a natural sum, without first defining what is meant by the sum of a divergent series, which led to confusing and contradictory results, until A. L. Cauchy gave a rigorous definition of the sum of a (convergent) series. In 1890, E. Ces\' aro realized that one could give a rigorous definition of the sum of some divergent series. In the years after, several other mathematicians gave other definitions of the sum of divergent series, although these are not always compatible: different definitions can give different answers for the sum of the same divergent series; so, when talking about the sum of a divergent series, it is necessary to specify which summation method we are using \cite{Wi}. A historical turning point in the study of divergent series was the publication of the G. H. Hardy monograph Divergent series  \cite{H}. At present, there are many publications in this field, including their applications \cite{Be}, \cite{M}. In this paper we have applied the Euler summation method of divergent series.  Algorithms were implemented in the wxMaxima. The wxMaxima program is an interface for working with the freely downloadable open source CAS (computer algebra system) program Maxima. In order to be able to apply this method in all cases, we have established and proved one new theorem (Theorem \ref{th:moj}) and given its consequence.
This method, like the method of undetermined coefficients, can be used only in special cases, if the right-hand side of the differential equation is typical, i.e. it is a constant, a polynomial function, exponential function $e^{\alpha x}$, sine or cosine functions $\sin{\beta x}$ or $\cos {\beta x}$, or finite sums and products of these functions (with constants $\alpha$ and $\beta$).

\section{Maclaurin expansion of a rational function}
We are interested in the Maclaurin series of a rational function.
\begin{theorem}\label{th:1} Let us consider a rational function which is defined at zero, with the conventional normalization
\begin{equation}\label{eq:rf}
R_{[L/M]}( t) = \frac{p_{0} + p_{1}t + \cdots + p_{L}t^{L}}{1 + q_{1}t + \cdots + q_{M}t^{M}}
\end{equation}
Let
\begin{equation}\label{eq:Ms}
C(t) = c_{0} + c_{1}t + c_{2}t^{2} + \cdots
\end{equation}
be a Maclaurin series of the rational function (\ref{eq:rf}). Then
\begin{equation}\label{eq:koef}
c_{n} = \mbf{q} \cdot\left( c_{n - M},c_{n - M + 1},\ldots,c_{n - 1} \right) + p_{n},\quad  n = 0,1,2,\ldots\
\end{equation}
where
$\mbf{q} =\left(-q_{M},- q_{M - 1},\ldots,- q_1 \right),$ $c_{n} = 0$ for 
$n < 0$ and $p_n = 0$ for $n > L$ and product 
$\mbf{q} \cdot\left( c_{n - M},c_{n - M + 1},\ldots,c_{n - 1} \right) $
is a dot product.
\end{theorem}
\begin{proof}
The statement follows from the identity
\begin{equation}\label{eq:aa}
p_0 + p_1t + \cdots + p_Lt^L = \left( 1 + q_1t + \cdots + q_Mt^M \right)\left( c_0 + c_1t + c_2t^2 + \cdots \right)
\end{equation}
after expanding the right-hand side (\ref{eq:aa}) and comparing coefficients for the
same powers of \(t\) we get
\[ q_Mc_{n - M} + q_{M - 1}c_{n - M + 1} + \cdots + q_{1}c_{n - 1}+c_n = p_{n}\]
From this it immediately follows  that
\[\begin{array}{r@{\,=\,}l}
c_{n} & - q_{M}c_{n - M} - q_{M - 1}c_{n - M + 1} - \cdots - q_{1}c_{n - 1} + p_{n}\\
& \left(-q_{M}, -q_{M - 1},\ldots, -q_{1} \right) \cdot \left( c_{n - M}, c_{n - M + 1},\ldots,\ c_{n - 1} \right) + p_{n}\\
& \mbf{q} \cdot \left(c_{n - M},c_{n - M + 1},\ldots,c_{n - 1} \right) + p_{n}
\end{array}\]
$n = 0,1,2,\ldots$
\end{proof}

\begin{example}\label{ex:fra}
Find the coefficients of the Maclaurin series of the function 
\begin{equation}\label{eq:ab}
R_{[2/4]}(t) = \frac{1 + t + 4t^{2}}{1 - 2t + 2t^{2} + 4t^{3} + 4t^{4}}
\end{equation}
\end{example}
We have $\mbf{q}=(-4, - 4, - 2, 2)$. Then
\[\begin{array}{r@{\,=\,}l}
c_{0}& \mbf{q} \cdot \left(c_{-4},c_{-3},c_{-2},c_{-1} \right) + p_{0} = (-4,-4,-2,2) \cdot (0,0,0,0) + 1 = 1\\[0.5ex]
c_{1}& \mbf{q} \cdot \left(c_{-3},c_{-2},c_{-1},c_{0} \right) + p_{1} =
 (-4,-4,-2,2) \cdot (0,0,0,1) + 1 = 3\\[0.5ex]
c_{2}& \mbf{q} \cdot \left(c_{-2},c_{-1},c_{0},c_{1} \right) + p_{2} = \mbf{q} \cdot ( 0,0,1,3) + 4 = 8\\[0.5ex]
c_{3}&6,\quad c_{4} = -20,\quad c_{5} = -96, \quad c_{6} = -208, \quad 
c_{7} = - 168, \quad c_{8} = 544
\end{array}\]

We have the Maclaurin series of (\ref{eq:ab}):
\begin{equation}\label{eq:ac}
\begin{array}{cl}
&\displaystyle{\frac{1 + t + 4t^{2}}{1 - 2t + 2t^{2} + 4t^{3} + 4t^{4}}}\\[1.8ex]
= 1 + 3t + 8t^{2} + &6t^{3} - 20t^{4} - 96t^{5} - 208t^{6} - 168t^{7} + 544t^{8} + \cdots
\end{array}
\end{equation}

\begin{remark}\label{rema} 
It is interesting to note that the sequence $\{c_n\}$ in (\ref{eq:Ms}) is homogeneous linear recurrence with constant coefficients with initial conditions.  
\end{remark}
\begin{lemma}\label{polkon} Let $t_{1},t_{2},\ldots,t_{M}$ be roots of the denominator in (\ref{eq:rf}). Then, the series (\ref{eq:Ms}) converges for all  $t\in \R$ for which
\begin{equation}
\left| t \right| < r = \min\{\left| t_{1} \right|,\left| t_{2} \right|,\ldots,{|t}_{M}|\}.
\end{equation}
\end{lemma}
\begin{example}\label{ex:frb}
The series in (\ref{eq:ac}) converges for $t\in \R$ for which
$\left| t \right| < \frac{\sqrt{3 - \sqrt{5}}}{2}.$
\end{example}

\section{Pad\' e approximant}
 
A Pad\' e approximant~is the ``best" approximation of a function by a rational function~of given order. Pad\' e approximants are usually superior to Maclaurin series when functions contain poles, because the use of rational functions allows them to be well-represented. It often gives better approximation of the function than truncating its Maclaurin series, and it may still work where the Maclaurin series does not converge.

In \cite{Ba}, \cite{W} an algorithm to determine the Pad\' e approximant of functions that are expressed by Maclaurin expansion is described. This technique can be used for finding a rational function if we know its Maclaurin expansion. A necessary condition for this is to know the upper estimate (as small as possible) of the degrees of a polynomial in the numerator and the denominator. In general, without knowledge of this requirement, an accurate estimate of the rational function is not possible. 
Incorrectly estimating at least one of the degrees of polynomials, we would only approximate the searched for rational function.
However, when accurately estimating the degree of the numerator and the denominator, or even when estimating when at least one degree would exceed the true value, we always get the same accurate estimate of the sought after rational function, whose Maclaurin expansion we know.
If we estimate the degree of the numerator $L$ and the denominator $M$, then to determine the rational function we need a Maclaurin expansion to the degree $L + M$, inclusive. 

 There are many effective methods for determining the rational function of its Maclaurin expansion. One of them is described in the mentioned works \cite{Ba}, \cite{W}:

\begin{equation}\label{eq:acc}
R_{[L/M]}( t ) = \frac{\left| \begin{matrix}
c_{L - M + 1} & c_{L - M + 2} & \cdots &c_L & c_{L + 1} \\
 \vdots & \vdots & \ddots & \vdots \\
c_{L} & c_{L + 1} & \cdots &c_{L+M-1}& c_{L + M} \\
\sum\limits_{j = M}^{L}{c_{j - M}t^{j}} &
 \sum\limits_{j = M - 1}^{L}{c_{j - M + 1}t^{j}} & \cdots & 
 \sum\limits_{j = 1}^{L}{c_{j-1}t^{j}}&\sum\limits_{j = 0}^{L}{c_{j}t^{j}} \\
\end{matrix} \right|}{\left| \begin{matrix}
c_{L - M + 1} & c_{L - M + 2} & \cdots &c_L &  c_{L + 1} \\
 \vdots & \vdots & \ddots & \vdots \\
c_{L} & c_{L + 1} & \cdots &c_{L+M-1} &  c_{L + M} \\
t^{M} & t^{M - 1} & \cdots &t & 1 \\
\end{matrix} \right|}
\end{equation}
where \(c_{n} = 0\ \)for \(n < 0\) and \(q_{n} = 0\) for \(n > M.\)

The element in the last row, in the $k$-th column, in the numerator is $\sum\limits_{j=M-k+1}^Lc_{j-M+k-1}t^j.$

Let us interpret the described algorithm on the series on the right side in (\ref{eq:ac})
\begin{equation}\label{eq:ad}
1 + 3t + 8t^{2} + 6t^{3} - 20t^{4} - 96t^{5} - 208t^{6} - 168t^{7} + 544t^{8} + \cdots \end{equation}
Suppose that we have estimated the degree of the polynomial of the denominator $L = 4$ and numerator $M = 3$.

Then using (\ref{eq:acc}) we have
\[R_{[3/4]}(t ) = \frac{\left| \begin{matrix}
1 & 3 & 8 & 6 & - 20 \\
3 & 8 & 6 & - 20 & - 96 \\
8 & 6 & - 20 & - 96 & - 208 \\
6 & - 20 & - 96 & - 208 & - 168 \\
0 & t^{3} & t^{2} + 3t^{3} & t + 3t^{2} + 8t^{3} & 1 + 3t + 8t^{2} + 6t^{3} \\
\end{matrix} \right|}{\left| \begin{matrix}
1 & 3 & 8 & 6 & - 20 \\
3 & 8 & 6 & - 20 & - 96 \\
8 & 6 & - 20 & - 96 & - 208 \\
6 & - 20 & - 96 & - 208 & - 168 \\
t^{4} & t^{3} & t^{2} & t & 1 \\
\end{matrix} \right|}\]
\begin{equation}\label{eq:add}
= \frac{14400t^{2} + 3600t + 3600}{14400t^{4} + 14400t^{3} + 7200t^{2} - 7200t + 3600} = \frac{4t^{2} + t + 1}{4t^{4} + 4t^{3} + 2t^{2} - 2t + 1}
\end{equation}

We have got the same rational function as in (\ref{eq:ab}). We have estimated the degree of the polynomial of the numerator as 3. After the calculation we can see that it is only 2. Note that we would get the same result for $L = 2$, but also for $L>2$ and $M> 4$. When estimating $L<2$ or
$M<4$, we get only an approximate estimate of the rational function.

\section{Differential operator and matrix differential operator}

It is sometimes convenient to adopt the notation
$Dy$, $D^{2}y$, $D^{3}y,\ldots, D^{n}y$ to denote $\displaystyle\frac{dy}{dx}$,  $ \displaystyle\frac{d^{2}y}{dx^{2}}$, $\displaystyle\frac{d^{3}y}{dx^{3}},\cdots,\displaystyle\frac{d^{n}y}{dx^{n}}$. The symbols $Dy$, $D^{2}y,\ldots$ are called {\it{differential operators}} \cite{Ch}, \cite{F1} and have properties analogous to those of algebraic quantities \cite{Hu}.

Problems of the operator calculus for solving linear differential equations are well dealt with in several publications, e.g., in publications \cite{Ch}, \cite{Hu}, \cite{S}. 

\noindent Using the operator notation, we shall agree to write the differential equation
\begin{equation}\label{eq:DR} 
\left(a_{n}y^{(n)} + a_{n - 1}y^{(n - 1)} + \cdots + a_{1}y^{\prime} + a_{0}\right)y = f(x),\quad a_n \neq 0,\ a_i\in \mathbb{R}
\end{equation}
as
\begin{equation}\label{eq:DRD}
\left(a_{n}D^{n} + a_{n - 1}D^{n - 1} + \cdots + a_{1}D + a_{0}I\right)y = f(x)
\end{equation}
where $I$ is the identity operator.  
\begin{remark}The identity operator $I$ maps a real number to the same real number $Ir=r$ \cite{W1}. To simplify writing, we will omit it in the following formulas.
\end{remark}

We will also use a concise notation for (\ref{eq:DRD}).
\begin{equation}\label{eq:DRDp}
\phi(D)y = f(x),
\end{equation}
where
\begin{equation}\label{eq:DO}
\phi(D) = a_nD^n + a_{n - 1}D^{n - 1} + \cdots + a_{1}D + a_{0}
\end{equation}
is called an \emph{operator polynomial} in \(D.\)
If we want to emphasize the degree of the polynomial operator (\ref{eq:DO}), we shall
write it in the form \(\phi_{n}( D )\).
\begin{definition}
If \(S = \{\mbf{v}_{1},\mbf{v}_{2},\ldots,\mbf{v}_{n}\}\)
is a set of vectors in a vector space \(V,\) then the set of all linear
combinations of
\(\mbf{v}_{1},\mbf{v}_{2}, \ldots,\mbf{v}_{n}\) is called the \emph{span} of
\(\mbf{v}_{1},\mbf{v}_{2},\ldots,\mbf{v}_{n}\) and is denoted by
\(\textrm{span}(\mbf{v}_{1},\mbf{v}_{2},\ldots,\mbf{v}_{n})\)
or \textrm{span}(S).
\end{definition}
Let \(G\) be a~vector space of all differentiable functions. Consider the
subspace \(V \subset G\) given by
\begin{equation}\label{eq:spanfi}
V=\textrm{span}\big( f_{1}(x),f_{2}(x),\ldots,f_{n}(x) \big),
\end{equation}
where we assume that functions \(f_{1}(x),f_{2}(x),\ldots,f_{n}(x)\)
are linearly independent. Since the set 
\(B = \{ f_{1}(x),f_{2}(x),\ldots,f_{n}(x)\}\) is linearly independent, it is a~basis for V.\par
The functions \(f_{i}(x),\ i = 1,2,\ldots,n\) expressed in basis \(B\) using base vector coordinates are usually written
\[\left\lbrack f_{1}(x)\right\rbrack_{B} = \begin{bmatrix}
1 \\
0 \\
\vdots \\
0 \\
\end{bmatrix},\ \left\lbrack f_{2}(x)\right\rbrack_{B} = \begin{bmatrix}
0 \\
1 \\
\vdots \\
0 \\
\end{bmatrix},\ldots,\left\lbrack f_{n}(x) \right\rbrack_{B} = \begin{bmatrix}
0 \\
0 \\
 \vdots \\
1 \\
\end{bmatrix}\]
The vector \(\left\lbrack f_{i}(x) \right\rbrack_{B}\) has in the \emph{i-}th row 1 and~0 otherwise.

Further, assume that the differential operator \( D\) maps \(V\)  into itself.\\
 Let 
\[D( f_{i}( x) )= \sum_{j = 1}^{n}c_{ij}f_{j}( x),\quad  i = 1,2,\ldots,n,\]
where \(c_{ij} \in \mathbb{R}\),  \(i,j = 1,2,\ldots,n\) are constants.
Then
\[\left\lbrack D\left( f_{i}(x) \right) \right\rbrack_{B} = \begin{bmatrix}
c_{i1} \\
c_{i2} \\
 \vdots \\
c_{in} \\
\end{bmatrix},\quad  i = 1,2,\ldots,n \]
and (see \cite{P})
\begin{equation}\label{eq:Poole}
\begin{array}{r@{\,=\,}l}
\left\lbrack D \right\rbrack_{B} & \left\lbrack \left\lbrack D\left( f_{1}(x) \right) \right\rbrack_{B} \Shortstack{. . . .}\left\lbrack D\left( f_{2}(x) \right) \right\rbrack_{B} \Shortstack{. . . .}\cdots  \Shortstack{. . . .}\left\lbrack D\left( f_{n}(x) \right) \right\rbrack_{B} \right\rbrack\\ [2ex]
& 	
\begin{bmatrix}
	c_{11} & c_{21} & \cdots & c_{n1} \\
	c_{12} & c_{22} & \cdots & c_{n2} \\
	\vdots & \vdots & \ddots & \vdots \\
	c_{1n} & c_{2n} & \cdots & c_{nn} \\
\end{bmatrix}
\end{array}
\end{equation}
If
\[f(x) = \sum_{i = 1}^{n}{\alpha_{i}f_{i}(x),\quad \alpha_{i} \in \mathbb{R},\quad i = 1,2,\ldots,n},\quad  f(x) \in V \]
then
\[\left\lbrack f(x) \right\rbrack_{B} = \begin{bmatrix}
\alpha_{1} \\
\alpha_{2} \\
 \vdots \\
\alpha_{n} \\
\end{bmatrix}\]
We express the derivative of the function \(f(x)\):
\[Df(x) = \sum_{i = 1}^{n}{\alpha_{i}Df_{i}(x)} = \sum_{i = 1}^{n}\alpha_{i}\sum_{j = 1}^{n}{c_{ij}f_{j}(x)} = \sum_{j = 1}^{n}{\sum_{i = 1}^{n}\alpha_{i}c_{ij}f_{j}(x)}\]
respectively
\[\left\lbrack D( f(x)) \right\rbrack_{B}= \begin{bmatrix}
\sum\limits_{i=1}^n c_{i1}\alpha_{i} \\
\sum\limits_{i=1}^n c_{i2}\alpha_{i} \\
 \vdots \\
\sum\limits_{i=1}^n c_{in}\alpha_{i} \\
\end{bmatrix}=\begin{bmatrix}
c_{11} & c_{21} & \cdots & c_{n1} \\
c_{12} & c_{22} & \cdots & c_{n2} \\
 \vdots & \vdots & \ddots & \vdots \\
c_{1n} & c_{2n} & \cdots & c_{nn} \\
\end{bmatrix}\begin{bmatrix}
\alpha_{1} \\
\alpha_{2} \\
 \vdots \\
\alpha_{n} \\
\end{bmatrix}\]
Let us further simply and denote \(\left\lbrack D \right\rbrack_{B}\) as  \(\mathcal{D}_B\). 

The matrix \(\mathcal{D}_B\) we will call \emph{the matrix differential operator  corresponding to a vector space} \(V\)  \emph{with the considered basis} \(B\). 

Denote  
\[\left\lbrack\mathcal{D}(f(x) \right\rbrack_{B}=\mbf{f}_B^\prime = \begin{bmatrix}
\beta_{1} \\
\beta_{2} \\
\mbf{\vdots} \\
\beta_{n} \\
\end{bmatrix}\quad\textrm{and}\quad \left\lbrack f(x) \right\rbrack_{B} = \mbf{f}_B\]
then
\begin{equation}\label{eq:lint}
\mbf{f}_B^\prime=\mathcal{D}_B\mbf{f}_B
\end{equation}
Note that the matrix transformation \(\mathcal{D}_B:V \rightarrow V\) defined by (\ref{eq:lint})
is a linear transformation.

As mentioned in (\ref{eq:Poole}), the $i$-th, \(\left( i = 1,2,\ldots,n \right)\ \)column of the matrix \(\mathcal{D}\) expresses the derivative of the function \(f_{i}(x).\)

\begin{definition} Let
\({1}/{\phi(D)}f(x)\) (or \(\phi^{- 1}(D)f(x)\)) be defined as a~particular solution \(y_{p}\) of the differential equation (\ref{eq:DRD}) such that 
\(\phi(D)y_{p} = f(x).\) We call \({1}/{\phi(D)}\) the \emph{inverse differential operator} to \(\phi(D)\) \cite{S}. Analogously we define  an inverse matrix differential operator: Let \(\mathcal{D}_B\) 
be a matrix differential operator corresponding to a vector space  \(V\)  with the considered basis \(B\) for an equation (\ref{eq:DRD}) with right-side \(\mbf{f}\). Let \(\phi^{-1}\left(\mathcal{D}\right)\mbf{f}\) be defined as a~particular solution \(\mbf{y}_{p}\) of the differential equation (\ref{eq:DRD}), such that \(\phi(\mathcal{D})\mbf{y}_{p}=\mbf{f}.\) Then we call \(\phi^{-1}(\mathcal{D})\) an \emph{inverse matrix differential operator} to \(\phi(\mathcal{D})\). 
\end{definition}

\section{Inverse of a differential operator and action on~a~continuous function}

Now, we will deal with the evaluation of

\[\frac{1}{\phi\left( D \right)}f(x)\]
only in the case if $1/{\phi(D)}$ is expressed in ascending powers of \(D\). We assume that \(f(x)\) is a differentiable function. For example

\[\begin{array}{r@{\,=\,}l}
\displaystyle\frac{1}{1 - D}e^{\frac{x}{2}} & \left( 1 + D + D^{2} + \cdots \right)e^{\frac{x}{2}} = e^{\frac{x}{2}} + \frac{1}{2}e^{\frac{x}{2}} + \frac{1}{4}e^{\frac{x}{2}} + \frac{1}{8}e^{\frac{x}{2}} + \cdots \\[1ex]
& e^{\frac{x}{2}}\left( 1 + \frac{1}{2} + \frac{1}{4} + \frac{1}{8} + \cdots \right) = 2e^{\frac{x}{2}}
\end{array}\]

Really \(y = 2e^{\frac{x}{2}}\) is the particular solution of the differential equation
\[\left( 1 - D \right)y = e^{\frac{x}{2}}\]
In this case there is no problem as the series
\(1 + \displaystyle\frac{1}{2} + \displaystyle\frac{1}{4} + \displaystyle\frac{1}{8} + \cdots\) converges.

\subsection{The Ces{\'a}ro summability of a series}

Let \(\sum\limits_{n = 0}^{\infty}c_{n}\) be a number series, and let
\[s_{k} = c_{0} + c_{1} + c_{2} + \cdots + c_{k}\] be its \(k\)-th partial sum. The series 
\(\sum\limits_{n = 1}^{\infty}c_{n}\)
is called Ces{\'a}ro summable (or summable by arithmetic means), with Ces{\'a}ro sum \(A\in \R\), if
\[\lim_{n \rightarrow \infty}\frac{1}{n+1}\sum_{k = 0}^{n}s_{k} = A.\]
For example
\[\begin{array}{r@{\,=\,}l}
\displaystyle\frac{1}{1 + D}e^{x} & \left( 1 - D + D^{2} -  D^{3} + \cdots \right)e^{x} = e^{x} - e^{x} + e^{x} - e^{x} + \cdots\\[0.5ex]
 & e^{x}(1 - 1 + 1 - 1 + \cdots)
\end{array}\]
The series \(1 - 1 + 1 - 1 + \cdots\) does not converge. But it is Ces{\'a}ro summable, because
\[\lim_{n \rightarrow \infty}\frac{1}{n+1}\sum_{k = 0}^{n}s_{k} = \frac{1}{2}.\]
Hence
\[ \frac{1}{1 + D}e^{x} = \frac12e^{x}\]
Indeed, $y= \frac12e^{x}$ is the particular solution of the equation $(D+1)y=e^x$. 

\subsection{The Euler method} \cite{H}
If the numerical series \(\sum\limits_{n = 0}^{\infty}{c_{n}t^{n}\ }\)is convergent for small \(t,\) and defines a function \(f(t)\) of the complex variable \(t,\) one-valued and regular in an open and connected region containing the origin and the point \(t = 1\) and \(f( 1) = s,\) then we call \( s\) the \(\mathfrak{E}\) sum of \(\sum\limits_{n = 0}^{\infty}c_{n}\). 

In our case we take for the region the maximum domain of \(f(t)\). If \(\sum\limits_{n = 0}^{\infty}{c_{n}\ }\) is summable by the Euler method \(\mathfrak{E}\) to \(s\), we will denote 
\(\sum\limits_{n = 0}^{\infty}{c_{n}\ } = s\left( \mathfrak{E} \right).\)

All series that are summable by the Ces\'aro method to \(s\) are summable to the same value \(s\) by the Euler method.

Yet another simple example
\[\begin{array}{r@{\,=\,}l}
\displaystyle\frac{1}{1 - 2D}e^{x} & \left( 1 + 2D + 4D^{2} + 8D^{3} + \cdots \right)e^{x} = e^{x} + 2e^{x} + 4e^{x} + 8e^{x} + \cdots \\[0.5ex]
& e^{x}(1 + 2 + 4 + 8 + \cdots)
\end{array}\]

The numerical series \(1 + 2 + 4 + \cdots +2^n+ \cdots\) diverges. Let us create the power series

\begin{equation}\label{eq:af}
1 + 2t + 4t^{2} + \cdots + 2^nt^n + \cdots
\end{equation}

The series (\ref{eq:af}) defines a single-valued and analytic function on the
region \(\left| t \right| < \frac{1}{2}\) containing the origin

\[f\left( t \right) = \frac{1}{1 - 2t} = 1 + 2t + 4t^{2} + 8t^{3} + \cdots\]

The function defined in the domain \(\left| t \right| < \displaystyle\frac{1}{2}\) can be extended to the function \(\displaystyle\frac{1}{1 - 2t}\) by means of an analytic continuation defined on
\(\mathbb{C\ \backslash}\left\{ - \frac{1}{2} \right\}\). This extended function is always given unambiguously.

Since

\[f\left( 1 \right) = \frac{1}{1 - 2 \cdot 1} = - 1\]
We have
\[1 + 2 + 4 +\cdots + 2^n + \cdots = - 1\left( \mathfrak{E} \right).\]
Then
\[\frac{1}{1 - 2D}e^{x} = - e^{x}.\]
Indeed, \(y=-e^{x}\) is the particular solution of the differential equation\linebreak  (\(1 - 2D)y = e^{x}.\) 

\section{Using a divergent series for finding a particular solution of an
  ordinary nonhomogeneous linear differential equation with constant
  coefficients}

In this section we will try to explain the basic principles for finding the particular solution of an ordinary nonhomogeneous linear differential equation with constant coefficients with a special type of
right-hand side using the Euler method of summable divergent series.

\begin{lemma}\label{lb} \cite{F2}, \cite{L} If the Taylor series expansion for \(f(x)\) about the origin

\[f(x) = \sum_{k = 0}^{\infty}\frac{f^{\left( k \right)}(0)}{k!}x^{k}\]
converges for all \(x\), for which \(\left| x \right| < r\), then the
matrix series

\[f\left( A \right) = \sum_{k = 0}^{\infty}\frac{f^{\left( k \right)}(0)}{k!}A^{k}\]
converges for those \(A\), for which \(\rho\left(A\right) < r\), where
\(\rho\left(A\right)\) is the spectral radius of the matrix
\({A.}\) As a convention \(A^{0} = I,\) where \(I\)is the identity matrix.
\end{lemma}

Let us consider the differential equation
\begin{equation}\label{eq:ba}
\phi_{n}\left( D \right)y = g(x)
\end{equation}
Let
\[V=\textrm{span}\left( f_{1}(x),f_{2}(x),\ldots,f_{m}(x \right))\]
be a vector space of differentiable functions with the basis $B$ of $V$
\begin{equation}\label{eq:bb}
B=\left\{ f_{1}(x),f_{2}(x),\ldots,f_{m}(x) \right\}
\end{equation}

and let for every function $f(x) \in V$ be $f^{\prime}(x) \in V$.

 Assume that \(g(x) \in V\). The Maclaurin expansion
\begin{equation}\label{eq:bc}
\frac{1}{\phi_{n}(D)} = c_{0} + \sum_{k = 1}^{\infty}{c_{k}D^{k}}, c_{0}\ne 0
\end{equation}
converges for \(|D|<r,\, 0<r<\infty\).  Note that $c_0$ must be different from zero. Otherwise, the value 
$1/\phi_n(0)$ would not be defined.

If \(t\in \R\) is a parameter, then
\begin{equation}\label{eq:bd}
\frac{1}{\phi_{n}(tD)} = \sum_{k = 0}^{\infty}c_{k}t^{k}D^{k}
\end{equation}
The series (\ref{eq:bd}) converges depending on \(t\). If \(t = t_{0}\), then it converges for\linebreak 
\(| D | < \displaystyle\frac{r}{|t_{0}|}.\)

Denote by \(\mathcal{D}\) (of the type \(m \times m\)) the matrix differential operator corresponding to the basis $B$ in (\ref{eq:bb}), and let the matrix \(\phi_{n}\left( \mathcal{D} \right)\) be a regular matrix. Then, the matrix series (due to the Lemma \ref{lb})
\begin{equation}\label{eq:bee}
\left(\phi_{n}(\mathcal{D}) \right)^{-1} = c_{0}\mbf{I}_{m} + \sum_{k = 1}^{\infty}{c_{k}\mathcal{D}^{k}}
\end{equation}
converges only if the spectral radius \(\rho\left( \mathcal{D} \right) < r.\) If \( r \le \rho(A)<\infty\) then we put for \(\mathcal{D}\) in (\ref{eq:bee}) the matrix \(t\mathcal{D,}\) where
\[0<t < \frac{r}{\rho\left( \mathcal{D} \right)}\]
for which the spectral radius
\[\rho\left( t\mathcal{D} \right) = t\ \rho\left( \mathcal{D} \right) < r.\]
Using this in Lemma \ref{lb} we get that the matrix series with parameter \(t\)
\begin{equation}\label{eq:bf}
\left( \phi_{n}\left( t\mathcal{D} \right) \right)^{- 1} = \sum_{k = 0}^{\infty}{c_{k}{t^{k}\mathcal{D}}^{k}}
\end{equation}
converges for \(\rho\left( t\mathcal{D} \right) < r.\) Now we use the inverse matrix formula based on the conjugate matrix and we get that the elements of the matrix \(\left( \phi_{n}\left( t\mathcal{D} \right) \right)^{- 1}\) will be rational functions of \(t\) defined at \(t = 0\) and at \(t = 1\). 
We get \(\left( \phi_{n}\left( \mathcal{D} \right) 
\right)^{- 1}\mbf{g}\) the particular solution of (\ref{eq:ba}), where \(\mbf{g}=g(x)_B\).

Multiplying equation (\ref{eq:bf}) from the right-hand side by the vector \(\mbf{g,\ }\)we have
\begin{equation}\label{eq:bg}
\left( \phi_{n}\left( t\mathcal{D} \right) \right)^{- 1}\mbf{g} =\sum_{k = 0}^{\infty}{c_{k}{t^{k}\mathcal{D}}^{k}}\mbf{g}
\end{equation}
The left and right-hand sides of (\ref{eq:bg}) are matrices with \(m\) rows and one column.  Each row on the right-hand side is an infinite series of powers of
\(t,\) to which it corresponds on the left-hand side of (\ref{eq:bg}) in the same row to a rational function. 
Each rational function (on the left-hand side) is an analytic continuation (except for a finite number of points from the complex plane) to its maximum domain of the corresponding function expressed by a power series on the right-hand side, and all conditions for using the Euler summation method are satisfied. It follows for \(t=1{:}\)
\begin{equation}\label{eq:bgg}
\sum_{k = 0}^{\infty}c_{k}\mathcal{D}^{k}\mbf{g}=\left( \phi_{n}\left( \mathcal{D} \right) \right)^{- 1}\mbf{g}(\mathfrak{E})
\end{equation}
The formulas (\ref{eq:bg}), (\ref{eq:bgg}) represent one of the main fundamental results of the paper. 
\begin{remark}\label{LM}
If \(\mathcal{D}\) is a \(m\times m\) matrix and (\ref{eq:ba})  is a differential equation of order $n$ then from the left-side of (\ref{eq:bg}) follows that to find a rational function $R_{[L/M]}(t)$ using the Pad\' e approximant it is enough to take a truncated power series of degree $L+M$, where $L=n(m-1)$ and $M=nm$. 
(In a special case can be the final fraction in the form $R_{[(L-k)/(M-k)]}(t)$, if the fraction $R_{[L/M]}(t)$ is truncated by polynomial of the $k$-th degree.) 
\end{remark}
\begin{remark}
To calculate the right-hand side of (\ref{eq:bg}) using e.g., the software wxMaxima it is better to express (\ref{eq:bg}) in the form

\hspace{1mm}

\begin{center}
	c[0]$\cdot$g+sum(c[k]$\cdot$t\^{}k$\cdot$X:D.X,k,1,n)\\
\end{center}

\hspace{1mm}

\noindent The symbol of the assignment command is : (in wxMaxima) \cite{Ma}. The initial value X is g, $n$ is a finite upper bound of summation (see Remark \ref{LM}). 
\end{remark}

\begin{example}\label{ex:bii} Using the matrix differential operator and summable divergent series by the Euler method  find the particular solution of the differential equation
\begin{equation}\label{eq:bi}
\left( 1 - D - D^{2} \right)y = e^{x}\sin x - 2e^{x}\cos x
\end{equation}
\end{example} 
\noindent\emph{Solution.} We will solve the equation first using the matrix differential operator method \cite{F1}. From the method of undetermined coefficients it follows that the solution of the differential equation will belong to the vector space 
\[V=\textrm{span}\left( e^{x}\sin x,e^{x}\cos x \right)\]
with the basis
\[B=\left\{ e^{x}\sin x,e^{x}\cos x \right\}\]
The relevant matrix differential operator is
\[\mathcal{D} = \begin{bmatrix}
1 & - 1 \\
1 & 1 \\
\end{bmatrix}\]
The solution of the differential equation belongs to $V$.  We have to solve the matrix equation
\[\begin{array}{r@{\,=\,}l}
\left( \mbf{I}_{2} - \mathcal{D -}\mathcal{D}^{2} \right)\mbf{y}_p&
[ e^{x}\sin x - 2e^{x}\cos x]_B\\[1ex]
\begin{bmatrix}
0 & 3 \\
 - 3 & 0 \\
\end{bmatrix}\mbf{y}_p&\begin{bmatrix}
1 \\
 -2 \\
\end{bmatrix}
\end{array}\]

\begin{equation}\label{eq:bj}
\mbf{y}_{p}=\begin{bmatrix}
0 & 3 \\
 -3 & 0 \\
\end{bmatrix}^{- 1}\begin{bmatrix}
1 \\
 -2 \\
\end{bmatrix} = \begin{bmatrix}
0 & - \displaystyle\frac{1}{3} \\
\displaystyle\frac{1}{3} & 0 \\
\end{bmatrix}\begin{bmatrix}
1 \\
 - 2 \\
\end{bmatrix} = \begin{bmatrix}
\displaystyle\frac{2}{3} \\[1.2ex]
\displaystyle\frac{1}{3} \\
\end{bmatrix}
\end{equation}

The particular solution to equation (\ref{eq:bi}) is 
\[y_p = \frac{2}{3}\,e^{x}\sin x + \frac{1}{3}\,e^{x}\cos x\]

Now, we will look at solving this example from the point of view of divergent series. We express the particular solution  of the equation (\ref{eq:bi}) in the form
$$y_p=\frac{1}{1 - D - D^{2}}\,(e^{x}\sin x - 2e^{x}\cos x)$$

Let us expand in powers of the $D$ the inverse operator $\displaystyle\frac{1}{1 - D - D^{2}}$ using the method described in the beginning of the paper (Theorem \ref{th:1}) and let this expansion act on the right-hand side of equation (\ref{eq:bi}). We have
\[\begin{array}{r@{\,=\,}l} y_p&(1+D+2{{D}^{2}}+3{{D}^{3}}+5{{D}^{4}}+8{{D}^{5}}+\cdots)
(e^{x}\sin x - 2e^{x}\cos x)\\ [1ex]
&(e^x\sin x-2e^x\cos x)+(3e^x\sin x-e^x\cos x)+(8e^x\sin x+4e^x\cos x)\\ [1ex]
&+(6e^x\sin x+18e^x\cos x)+(-20e^x\sin x+40e^x\cos x)+\cdots\\ [1ex]
&(1+3+8+6-20+\cdots)e^x\sin x+(-2-1+4+18+\cdots)e^x\cos  x
\end{array}\] 
We need to determine the sum of the corresponding divergent series by the Euler method. However, we see that the presented calculation procedure is impractical. 
 
 We use the matrix representation using the right-hand side of (\ref{eq:bg}), where we determine the coefficients $c_k$ according to Theorem \ref{th:1}
 \[\mbf{f}(t) = \sum_{k=0}^{6}{(1,1)\cdot}\left(c_{n-2},c_{n-1} \right)t^{k}\mathcal{D}^k\mbf{g}\]
 where
 \(\mbf{g} =\begin{bmatrix} 1 \\  - 2 \\ \end{bmatrix}\),
 \(c_{- 1} = 0\), \(c_{0} = 1\). The upper limit of summation has been determined using Remark \ref{LM} ($m=2$, $n=2$, $M=m\cdot n=4$, $L=n\cdot (m-1)=2$, $L+M=6$). 
  
We get
\[\mbf{f}(t)=\begin{bmatrix}
1 + 3t + 8t^{2} + 6t^{3} - 20t^{4} - 96t^{5} - 208t^{6} + \cdots\\
 - 2 - t + 4t^{2} + 18t^{3} + 40t^{4} + 32t^{5} - 104t^{6} +\cdots \\
\end{bmatrix}\]

It follows from (\ref{eq:ad}), (\ref{eq:add}) that

\[1 + 3t + 8t^{2} + 6t^{3} - 20t^{4} - 96t^{5} - 208t^{6} +\cdots=\frac{4t^{2} + t + 1}{4t^{4} + 4t^{3} + 2t^{2} - 2t + 1}=f_{1}(t)\]
and
\[f_{1}(1) =\frac{2}{3}\]
So
\[1 + 3 + 8 + 6 - 20 - 96 - 208 + \cdots = \frac{2}{3}\mathfrak{(E)}\]

We will now find the Euler sum of the second divergent series
\[- 2 - 1 + 4 + 18 + 40 + 32 - 104 +\cdots \]

It is not our goal to present different types of Pad\' e approximant algorithms. Many computer algebra systems software include a procedure for calculating these. For example in Mathematica the procedure\\ 
PadeApproximant~[expr,{x,$x_0$,{m,n}}] gives the  approximant to expr about the point x=$x_0$, with numerator order m and denominator order n \cite{Wo}. \\
The wxMaxima procedure\newline  pade(taylor\_series, numer\_deg\_bound, denom\_deg\_bound) returns a list of all rational functions which have the given Taylor series expansion where the sum of the degrees of the numerator and the denominator is less than or equal to the truncation level of the power series, i.e. are "best" approximants, and which additionally satisfy the specified degree bounds. Where taylor\_series is a univariate Taylor series, numer\_deg\_bound and denom\_deg\_bound are positive integers specifying the degree bounds on the numerator and denominator \cite{Ma}.  

Here we use the procedure in wxMaxima (Figure \ref{fig:pa})
\begin{figure}[!hbp]
{\includegraphics[width=1\textwidth]{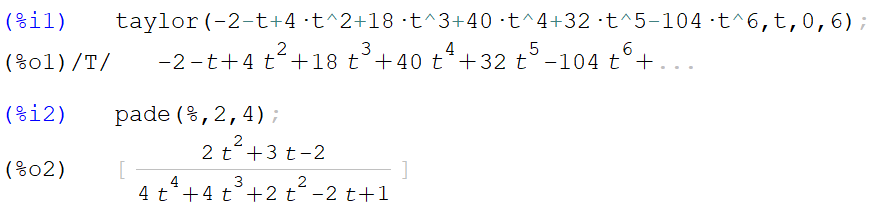}}
\caption{Pad\' e approximant in wxMaxima}\label{fig:pa}
\end{figure}

We have
\[- 2 - t + 4t^{2} + 18t^{3} + 40\ t^{4} + 32t^{5} - 104t^{6} + \cdots=\frac{2t^{2} + 3t - 2}{4t^{4} + 4t^{3} + 2t^{2} - 2t + 1}=f_{2}(t) \]

The value \(f_{2}(1)\), after the analytical continuation of \(f_{2}(t)\) to the whole\linebreak complex plane except for the four zero complex roots of the equation\linebreak 
\(4t^{4} + 4t^{3} + 2t^{2} - 2t + 1 = 0\), is
\[f_{2}(1) = \frac{1}{3}\]
So
\[- 2 - 1 + 4 + 18 + 40\  + 32 - 104 + \cdots = \frac{1}{3}\mathfrak{(E)}\]

We have obtained the same solution as in (\ref{eq:bj})
\[\mbf{f}\left( 1 \right) = \begin{bmatrix}
\displaystyle\frac{2}{3} \\[1.2ex]
\displaystyle\frac{1}{3} \\
\end{bmatrix}_B\]
\[y_p = \frac{2}{3}e^{x}\sin x + \frac{1}{3}e^{x}\cos x\]

If we want to calculate directly the appropriate rational functions (for control), we can use the idea of (\ref{eq:bf}). We get
\[\left( \mbf{I}_{2} - \mathcal{D}t - (\mathcal{D}{t)}^{2} \right)^{- 1}\begin{bmatrix}
1 \\
 - 2 \\
\end{bmatrix} = \begin{bmatrix}
\displaystyle\frac{4t^{2} + t + 1}{4t^{4} + 4t^{3} + 2t^{2} - 2t + 1} \\
\displaystyle\frac{2t^{2} + 3t - 2}{4t^{4} + 4t^{3} + 2t^{2} - 2t + 1} \\
\end{bmatrix}\]

Further we will need some theorems.
\begin{lemma}\label{chen}\cite{Ch}, \cite{Hu}, \cite{S}  Let $\phi_n(D)$ be a polynomial of $D$ of degree $n$. Then the nonhomogenous linear differential equation 
\begin{equation}\label{eq:bk}
\phi_n(D)y=Ae^{\alpha x}
\end{equation}
$A\in\R$, $\alpha$ is real or complex, has a particular solution 

\begin{displaymath}
y_p = \left\{ \begin{array}{ll}
\displaystyle\frac{Ae^{\alpha x}}{\phi_n(\alpha)} & \textrm{if  $\phi_n(\alpha) \ne 0$}\\
\displaystyle\frac{Ax^ke^{\alpha x}}{\phi_n^{(k)}(\alpha)} & \textrm{if $\phi_n^{(i)}(\alpha)=0$ 
for $i=0,1,2,\ldots,k-1$ but $\phi_n^{(k)}(\alpha)\ne 0$}
\end{array} \right.
\end{displaymath}

 where  $\phi_n^{(0)}(D)=\phi_n(D).$
\end{lemma}

\begin{theorem}\label{th:moj} Let $\phi_n(D)$ be a polynomial of $D$ of degree $n$. 
Let us consider the differential equation (\ref{eq:bk})  where 
$\phi_n^{(i)}(\alpha)=0$ for 
$i=0,1,2,\ldots,k-1$ but $\phi_n^{(k)}(\alpha)\ne 0$.
Then the particular solution of the differential equation 
\begin{equation}\label{eq:mja}
\phi_n^{(k)}(D)y=Ax^ke^{\alpha x}
\end{equation}
is also the particulation  solution of the equation (\ref{eq:bk}).
\end{theorem}
\begin{proof} For the proof of this theorem we need to prove two auxiliary relationships. Due to the operator shift theorem \cite{Ch}, \cite{F1}, \cite{S}  we have

\begin{itemize}
	\item[(1)]
	\[(D-\alpha)^k(Ax^ke^{\alpha x})=Ae^{\alpha x}(D+\alpha-\alpha)^kx^k=Ae^{\alpha x}k!\]
	\item[(2)] Let
	\[\begin{array}{r@{\,\,}l}
\phi_n(D)=&(D-\alpha)^k\phi_{n-k}(D),\quad\textrm{then}\\[1ex]
		{\phi^{(k)}_n(D)}\!=\!&
\left(\begin{matrix} k\\0 \end{matrix}\right) \left((D-\alpha)^k\right)^{(k)}\phi_{n-k}(D)\!+\!
\left(\begin{matrix} k\\ 1 \end{matrix}\right) \left((D-\alpha)^k\right)^{(k-1)}\phi^\prime_{n-k}(D)\!+\!\cdots\\ [1.8ex]
+&\left(\begin{matrix} k\\k \end{matrix}\right)(D-\alpha)^k \phi^{(k)}_{n-k}(D)
	\end{array}\]
\end{itemize}
hence
\[\phi^{(k)}_n(\alpha)=k!\phi_{n-k}(\alpha)\]
From (\ref{eq:mja}) we have
\begin{equation}\label{eq:rmja}
y=\frac{1}{\phi^{(k)}_n(D)}Ax^ke^{\alpha x}
\end{equation}
Substituting this in the (\ref{eq:bk}) we have
$$\phi_n(D)\frac1{\phi^{(k)}_n(D)}Ax^ke^{\alpha x}=A\frac{\phi_{n-k}(D)}{\phi^{(k)}_n(D)}
(D-\alpha)^{k}x^{k}e^{\alpha x}$$
$$=A k! \frac{\phi_{n-k}(D)}{\phi^{(k)}_n(D)}e^{\alpha x}=
A k! \frac{\phi_{n-k}(\alpha)}{\phi^{(k)}_n(D)}e^{\alpha x}= 
A k! \frac{\phi_{n-k}(\alpha)}{\phi^{(k)}_n(\alpha)}e^{\alpha x}=$$
$$=A k! \frac{\phi_{n-k}(\alpha)}{k! \phi_{n-k}(\alpha)}e^{\alpha x}=Ae^{\alpha x} $$
\end{proof}
\noindent This proves that (\ref{eq:rmja}) is also a particular solution of (\ref{eq:bk}). Note that the particulation solution of the differential equation (\ref{eq:mja}) may include the kernel of  $\phi_n(D)$ in (\ref{eq:bk}).

\begin{corollary}\label{cor:dosl}
Let $\phi_n(D)$ be a polynomial of $D$ of degree $n$. Let us consider the differential equation 
\begin{equation}\label{eq:jf1}
\phi_n(D)y=e^{\alpha x}(A \sin \beta x +B \cos \beta x)
\end{equation}
Let $\phi_n(\alpha+\beta i)=\phi^\prime_n(\alpha+\beta i)=\cdots=
\phi^{(k-1)}_n(\alpha+\beta i)=0$ but $\phi^{(k)}_n(\alpha+\beta i)\ne 0.$
Then, the particular solution of the differential equation 
\begin{equation}\label{eq:jf2}
\phi_n^{(k)}(D)y=x^{k}e^{\alpha x}(A \sin \beta x +B \cos \beta x)
\end{equation}
is also the particular  solution of the equation (\ref{eq:jf1}).
\end{corollary}

\begin{example} Determine a particular solution of the equation
\begin{equation}\label{eq:ca}
(D^2-4D+13)y=e^{2x}\left(4\sin 3x +2\cos 3x\right)
\end{equation}
\end{example}
\noindent\emph{Solution.}
First, we find the solution of the equation using a matrix differential operator \cite{F1}.
The roots of the characteristic equation
\(k^{2} - 4k + 13 = 0\) are \(2~+~3i,\,2~-~3i.\) Hence, the particular solution will be in the form
$$y=xe^{2x}\left(A\sin 3x +B\cos 3x\right)$$
This means that the particular solution (\ref{eq:ca}) belongs to the vector space
\[V =\textrm{span}\left(xe^{2x}\sin 3x,xe^{2x}\cos 3x,e^{2x}\sin 3x,e^{2x}\cos 3x\right)\]
with the basis
\[B = \left\{xe^{2x}\sin 3x,xe^{2x}\cos 3x,e^{2x}\sin 3x,e^{2x}\cos 3x\right\}\]
The matrix differential operator is 
\[\mathcal{D} = \begin{bmatrix}
2 & - 3 & 0 & 0 \\
3 & 2 & 0 & 0 \\
1 & 0 & 2 & -3 \\
0 & 1 & 3 & 2 \\
\end{bmatrix}\]
Then 
$$\mathcal{D}^2-4\mathcal{D}+13\mbf{I_4}=
\left[\begin{array}{cccc}
	0 & 0 & 0 & 0 \\
	0 & 0 & 0 & 0 \\
	0 & -6 & 0 & 0 \\
	6 & 0 & 0 & 0 \\
\end{array}\right]$$

The matrix $\mathcal{D}^{2} - 2\mathcal{D +}2\mbf{I}_{4}$ is singular. For this case it is possible to use for example the method of undetermined coefficients or the method described in  \cite{F1}. 

Now we shall find a particular solution to the equation (\ref{eq:ca}) using Corollary~\ref{cor:dosl}. We have to solve the equation
\begin{equation}\label{eq:cb}
(2D-4)\tilde{y}_p=xe^{2x}(4\sin 3x+2\cos 3x)
\end{equation}
The matrix differential operator $\mathcal{D}$ is the same. So we have to solve the matrix equation
$$ (2\mathcal{D}-4\mbf{I_4})\tilde{\mbf{y}}_p=[xe^{2x}(4\sin 3x+2\cos 3x)]_B$$
$$\begin{bmatrix}
0 & -6 & 0 & 0 \\
6 & 0 & 0 & 0 \\
2 & 0 & 0 & -6 \\
0 & 2 & 6 & 0 \\\end{bmatrix}
\tilde{\mbf{y}}_p=\begin{bmatrix}
4\\
2\\
0\\
0
\end{bmatrix}$$
$$\tilde{\mbf{y}}_p=\begin{bmatrix}
0 & -6 & 0 & 0 \\
6 & 0 & 0 & 0 \\
2 & 0 & 0 & -6 \\
0 & 2 & 6 & 0 \end{bmatrix}^{-1}\begin{bmatrix}
4\\
2\\
0\\
0
\end{bmatrix}\renewcommand{\arraystretch}{2}
=\begin{bmatrix}
0 & \displaystyle\frac{1}{6} & 0 & 0\\
-\displaystyle\frac{1}{6} & 0 & 0 & 0\\
\displaystyle\frac{1}{18} & 0 & 0 & \displaystyle\frac{1}{6}\\
0 & \displaystyle\frac{1}{18} & -\displaystyle\frac{1}{6} & 0
\end{bmatrix}\begin{bmatrix}
4\\
2\\
0\\
0
\end{bmatrix}=\begin{bmatrix}
\displaystyle\frac13\\[1ex] 
-\displaystyle\frac23\\[1ex]
\displaystyle\frac29\\[1ex]
\displaystyle\frac19
\end{bmatrix}
$$
The particulation solution of (\ref{eq:cb}) and also (\ref{eq:ca}) is 
\begin{equation}\label{eq:ce}
\tilde{y}_p=\frac{1}{3}xe^{2x}\sin{3x}-\frac{2}{3}xe^{2x}\cos{3x}+
\frac{2}{9}e^{2x}\sin{3x}+\frac{1}{9}e^{2x}\cos{3x}
\end{equation}
but $\frac29e^{2x}\sin 3x$ and $\frac19e^{2x}\cos 3x$ belong to the kernel of the operator \linebreak $D^2-4D+13$, then we can write the particular solution (\ref{eq:ce}) of the di\-fferential equation (\ref{eq:ca}) in the simpler form
\begin{equation}\label{eq:cd}
y_p=\frac13xe^{2x}\sin{3x}-\frac23xe^{2x}\cos{3x}
\end{equation}
Since the matrix $\mathcal{D}^{2} - 4\mathcal{D} +13\mbf{I}_{4}$ is singular, the idea described from (\ref{eq:ba}) to (\ref{eq:bgg}) for calculating a particular solution using divergent series cannot be used. However, if we continued to calculate the Pad\' e approximant (for example with the support of the open source software wxMaxima), we would get ``rational functions": substituting for $t = 1$ we get meaningless expressions. This is not the way how to find the correct result. 

Finally, we find a particular solution to the differential equation (\ref{eq:cb}) using the summation of divergent series by applying the Euler method. We will express
\[\begin{array}{r@{\,=\,}l} 
y_p&\displaystyle\frac1{2D-4}\left(4xe^{2x}\sin 3x+2xe^{2x}\cos 3x\right)=\\ [1.3ex]
 &\displaystyle{-\frac14\left(1+\frac{D}2+\frac{D^2}4+\frac{D^3}8+\cdots\right)}
 \left(4xe^{2x}\sin 3x+2xe^{2x}\cos 3x\right)\\
\end{array}\]
or better in the matrix form with powers of $t$
\begin{equation}
\mbf{y}_p=-\frac14\sum_{k=0}^\infty\frac{t^k}{2^k}\mathcal{D}^k\mbf{g}
\end{equation}
where $\mbf{g}=\begin{bmatrix} 4 & 2 & 0 & 0 \end{bmatrix}^T$

It follows from Remark \ref{LM} that it is enough to take the truncated power series of degree 7 and find rational functions  $R_{[L/M]}(t)$, where $L=3$ and $M=4$. 
We get 
\[\begin{bmatrix}-\frac{8659{{t}^{7}}}{256}-\frac{2449{{t}^{6}}}{64}-\frac{841{{t}^{5}}}{64}+\frac{59{{t}^{4}}}{16}+\frac{101{{t}^{3}}}{16}+\frac{11{{t}^{2}}}{4}-\frac{t}{4}-1\\[1.1ex]
	-\frac{3863{{t}^{7}}}{64}-\frac{379{{t}^{6}}}{128}+\frac{67{{t}^{5}}}{4}+\frac{359{{t}^{4}}}{32}+\frac{7{{t}^{3}}}{4}-\frac{19{{t}^{2}}}{8}-2t-\frac{1}{2}\\[1.1ex]
	-\frac{17143{{t}^{7}}}{128}-\frac{2523{{t}^{6}}}{64}+\frac{295{{t}^{5}}}{32}+\frac{101{{t}^{4}}}{8}+\frac{33{{t}^{3}}}{8}-\frac{{{t}^{2}}}{4}-\frac{t}{2}\\[1.1ex]
	-\frac{2653{{t}^{7}}}{256}+\frac{201{{t}^{6}}}{4}+\frac{1795{{t}^{5}}}{64}+\frac{7{{t}^{4}}}{2}-\frac{57{{t}^{3}}}{16}-2{{t}^{2}}-\frac{t}{4}
\end{bmatrix}\]

The first and second components are calculated in Figure \ref{fig:comp}.  Note that the third and fourth components of the particular solution do not have to be calculated at all, because this part of the solution belongs to the kernel of the operator $D^2-4D+13$, as we have already mentioned.
\begin{figure}[!tbh]
	{\includegraphics[width=0.8\textwidth]{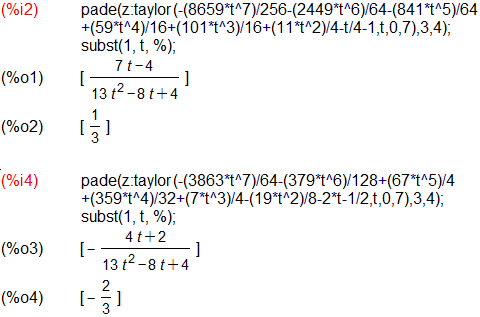}}
	\caption{Components of the solution - Euler summation method}\label{fig:comp}
\end{figure}
We get the same particular solution as in (\ref{eq:cd}).
So 
\[y_p=\frac13xe^{2x}\sin 3x-\frac23xe^{2x}\cos 3x\]

\section{Conclusion}

The content of this paper follows on from the paper \cite{F1}. It has a more or less theoretical character. It is not suitable for effectively finding a particular solution to a nonhomogeneous linear ODE with constant coefficients with a typical right-hand side. In the paper  we used the Euler method of summation of divergent series. The derived relations (\ref{eq:bg}) and (\ref{eq:bgg}) as well as Theorem \ref{th:moj} and its Corollary \ref{cor:dosl} and the algorithm for finding a particular solution using a matrix differential operator can be considered as the main results of the paper. We suggest that this method  with developed software applications could be used in courses teaching divergent series.

\end{document}